\documentclass[reqno]{amsart}
\usepackage{amsmath, amssymb, amsthm, epsfig}
\usepackage{hyperref, latexsym}
\usepackage{url} 
\def\today{\ifcase\month\or
  January\or February\or March\or April\or May\or June\or
  July\or August\or September\or October\or November\or December\fi
  \space\number\day, \number\year}

\DeclareMathOperator{\sgn}{\mathrm{sgn}}

 \newtheorem{theorem}{Theorem}
 \newtheorem{lemma}[theorem]{Lemma}
 \newtheorem{proposition}[theorem]{Proposition}
 \newtheorem{corollary}[theorem]{Corollary}
 \theoremstyle{definition}

 \theoremstyle{remark}
 \newtheorem{remark}[theorem]{Remark}
 \newcommand{\mc}{\mathcal}

 \newcommand{\C}{\mathbb{C}}
 \newcommand{\R}{\mathbb{R}}
 \newcommand{\N}{\mathbb{N}}
 
 \newcommand{\Z}{\mathbb{Z}}

\newcommand{\tQ}{\widetilde{Q}}

 \newcommand{\bx}{\boldsymbol{x}}

 \newcommand{\dx}{\text{\rm d}x}

\begin{document}

\title[Approximations to the Bernoulli periodic functions]{Sharp approximations to the Bernoulli periodic functions by trigonometric polynomials}
\author[E. Carneiro]{Emanuel Carneiro}
\thanks{The author was supported by CAPES/FULBRIGHT grant BEX 1710-04-4.}
\date{\today}
\address{Department of Mathematics, University of Texas at Austin, Austin, TX 78712-1082.}
\email{ecarneiro@math.utexas.edu}
\allowdisplaybreaks
\numberwithin{equation}{section}

\begin{abstract} 
We obtain optimal trigonometric polynomials of a given degree $N$ that majorize, minorize and approximate in $L^1(\R/\Z)$ the Bernoulli periodic functions. These are the periodic analogues of two works of F. Littmann (\cite{Lit} and \cite{Lit2}) that generalize a paper of J. Vaaler (\cite{V}). As applications we provide the corresponding Erd\"{o}s-Tur\'{a}n-type inequalities, approximations to other periodic functions and bounds for certain Hermitian forms.
\end{abstract}

\maketitle
\section{Introduction}%%%%%%%%%%%%%%%%%%%%%%%%%%%%section 1%%%%%%%%%%%%%%%%%%%%%%%%%%%%%%%%%%%%%%%%%%%%%%%%%%%%%%
An entire function $F(z)$ is said to be of exponential type $2\pi\delta$ if, for any $\epsilon >0$,
$$|F(z)| \leq A_{\epsilon} e^{|z|(2\pi \delta + \epsilon)}$$
for all complex $z$ and some constant $A_{\epsilon}$ depending on $\epsilon$. We denote by $E(2\pi \delta)$ the set of all functions of exponential type $2\pi \delta$ which are real on $\R$. Given a function $f:\R \to \R$, the extremal problem is the search for a function $F \in E(2\pi \delta)$ such that
\begin{align}\label{Intro1}
\begin{split}
 \textrm{(a)} &\ \ F(x) \geq f(x) \ \ \textrm{for all} \ \  x \in \R,  \\
 \textrm{(b)} &\ \ \int_{\R} (F(x) - f(x))\, \dx = \displaystyle\min_{\stackrel{G \in E(2\pi\delta)}{G\geq f}} \int_{\R} (G(x)-f(x))\, \dx 
\end{split}
\end{align}
A function $F \in E(2\pi\delta)$ satisfying (\ref{Intro1}) is called an extremal majorant of exponential type $2\pi\delta$ of $f$. Extremal minorants are defined analogously.\\

In the special case $f(x) = \sgn(x)$, an explicit solution to the problem (\ref{Intro1}) was found in the 1930's by A. Beurling, but his results were not published at the time of their discovery. Later, Beurling's solution was rediscovered by A. Selberg, who realized its importance in connection with the large sieve inequality in analytic number theory (see \cite{S}). An account of these functions, their history and many other applications can be found in the survey \cite{V} by J. Vaaler. \\

In the paper \cite{V}, Vaaler addresses the periodic analogue of the extremal problem (\ref{Intro1}): given a periodic function $\psi:\R/\Z \to \R$ and a nonnegative integer $N$, one wants to find a trigonometric polynomial $x \mapsto P(x;N)$ of degree at most $N$ such that
\begin{align}\label{Intro1.1}
\begin{split}
 \textrm{(a)} &\ \ P(x;N) \geq \psi(x) \ \ \textrm{for all} \ \  x \in \R/\Z,  \\
 \textrm{(b)} &\ \ \int_{\R/\Z} (P(x;N) - \psi(x))\, \dx  \ \ \textrm{is minimal} 
\end{split}
\end{align}
In this case $P(x;N)$ will be an extremal majorant. Extremal minorants are defined analogously. One can also consider the problem of best approximating $\psi(x)$ in the $L^1(\R/\Z)$-norm, i.e. we seek $P(x;N)$ such that the integral
\begin{equation}\label{Intro1.3}
 \int_{\R/\Z} |P(x;N) - \psi(x)|\, \dx 
\end{equation}
is minimal. The periodic version of the function $\sgn(x)$ appearing in \cite{V} is the sawtooth function $\Psi: \R/\Z \to \R$ defined by
\begin{equation}\label{Psi}
\Psi(x) =
\left\{
\begin{array}{lcc}
 x - [x] - 1/2 & \textrm{if} & x\notin \Z\\
 0 & \textrm{if}& x\in \Z
 \end{array}
 \right.
\end{equation}
where $[x]$ is the integer part of $x$. Vaaler then solves problems (\ref{Intro1.1}) and (\ref{Intro1.3}) for this function $\Psi(x)$, obtaining applications to the theory of uniform distribution (we return to this subject in Section 4).\\

Recently, F. Littmann extended the ideas of Beurling and Selberg to solve the extremal problem (\ref{Intro1}) for the functions $f(x) = \sgn(x)x^n$, $n \in \N_0 = \N \cup \{0\}$. He found not only the unique extremal majorants and minorants (see \cite{Lit}), but also the best entire approximation in the $L^1(\R)$-norm (see \cite{Lit2}) to these functions. The purpose of this paper is to transfer these two works of Littmann to periodic versions, namely to solve the extremal problems (\ref{Intro1.1}) and (\ref{Intro1.3}) for the Bernoulli periodic functions $\mc{B}_n(x)$, thus generalizing the periodic machinery developed by Vaaler in \cite{V}. We briefly outline our results below.\\

The Bernoulli polynomials $B_n(x)$ can be defined by the power series expansion
\begin{equation}\label{Bernoulli polynomials}
\dfrac{t e^{xt}}{e^t -1} = \sum_{n=0}^{\infty}\dfrac{B_n(x)}{n!}t^n
\end{equation}
where $|t| < 2\pi$, and the Bernoulli periodic functions $\mathcal{B}_n(x)$ by
\begin{equation} \label{Intro2}
\mathcal{B}_n(x) = B_n(x - [x])
\end{equation}
For $n\geq 1$, the Bernoulli periodic functions have the Fourier expansion
\begin{equation}\label{Bernoulli periodic}
\mathcal{B}_n(x) = -\dfrac{n!}{(2\pi i)^n} \sum_{\stackrel{k=-\infty}{k \neq 0}}^{\infty}\dfrac{1}{k^n} \, e(kx)
\end{equation}
Observe that, apart from a renormalization at $x=0$, we have 
\begin{equation*}
\mathcal{B}_1(x) = \Psi(x)
\end{equation*}
It is a well-known fact that the function $B_{2n}(x) \, (n\geq 1)$ has exactly one zero in the interval $(0,1/2)$, which we denote by $z_{2n}$ (put $z_0 = 0$). By a result of D.H. Lehmer (\cite{Leh}) the inequality
\begin{equation}\label{Sec2.1}
1/4- \pi^{-1}2^{-2n-1} < z_{2n} < 1/4 
\end{equation}
holds for $n\in \N$. The odd Bernoulli polynomials $B_{2n+1}(x)$ have zeros at $x=0$ and $x=1/2$, but no zeros in the interval $(0,1/2)$. Define two sequences $\{\alpha_n\}_{n \in \N_0}$ and $\{\beta_n\}_{n \in \N_0}$ by:
\begin{equation} \label{Intro3}
\begin{array}{cclcccl}
\alpha_{4k}&:=& 1-z_{4k}& & \beta_{4k}& := & z_{4k} \\
\alpha_{4k+1}&:=& 0& & \beta_{4k+1}& := & 1/2 \\
\alpha_{4k+2}&:=& z_{4k+2}& & \beta_{4k}& := & 1- z_{4k+2} \\
\alpha_{4k+3}&:=& 1/2 & & \beta_{4k+3}& := & 0 
\end{array}
\end{equation}
We point out that $B_{n+1}(x)$ assumes its maximum in $[0,1]$ at $x = \alpha_n$, and its minimum in $[0,1]$ at $x = \beta_n$ (cf. \cite[Lemma 5]{Lit}). Our first result is

\begin{theorem}\label{thm2.3}
Let $n$ and $N$ be nonnegative integers. There exist real valued trigonometric polynomials $x \mapsto P_{n+1}(x;N,\beta_n)$ and $x \mapsto P_{n+1}(x;N,\alpha_n)$ of degree at most $N$ such that 
\begin{equation}
P_{n+1}(x;N,\beta_n) \leq \mathcal{B}_{n+1}(x) \leq P_{n+1}(x;N,\alpha_n)
\end{equation}
at each point $x \in \R/\Z$. Moreover,\vspace{0.15cm}

\begin{itemize}
 \item[(i)] If $Q(x)$ is a real trigonometric polynomial of degree at most $N$ that satisfies $Q(x) \leq \mathcal{B}_{n+1}(x)$ for all $x \in \R/\Z$, then
\begin{equation}\label{Sec2.8.1}
\int_{\R/\Z} (\mathcal{B}_{n+1}(x) - Q(x))\, \dx \geq - \frac{B_{n+1}(\beta_n)}{(N+1)^{n+1}}
\end{equation}
with equality if and only if $Q(x) = P_{n+1}(x;N,\beta_n)$.\vspace{0.15cm}

\item[(ii)] If $\tQ(x)$ is a real trigonometric polynomial of degree at most $N$ that satisfies $\mathcal{B}_{n+1}(x) \leq \tQ(x)$ for all $x \in \R/\Z$, then
\begin{equation}\label{Sec2.8}
\int_{\R/\Z} (\tQ(x) - \mathcal{B}_{n+1}(x))\, \dx \geq  \frac{B_{n+1}(\alpha_n)}{(N+1)^{n+1}}
\end{equation}
with equality if and only if $\tQ(x) = P_{n+1}(x;N,\alpha_n)$. 
\end{itemize}

\end{theorem}

The extremal trigonometric polynomials $P_{n+1}(x;N,\beta_n)$ and $P_{n+1}(x;N,\alpha_n)$ are explicitly described in Section 2 (equations (\ref{Sec2.5.1}) - (\ref{Sec2.7})).

An interesting remark is that when $n$ is odd, the extremals of the Bernoulli polynomials $B_{n+1}(x)$ in $[0,1]$ are related to the Riemann zeta function (see \cite{Leh}) by 
\begin{align}\label{Riemann}
\begin{split}
B_{2k}(0) &= \dfrac{(-1)^{k-1} (2k)!}{2^{2k-1} \pi^{2k}} \zeta(2k) \\
B_{2k}(1/2) & =  -(1 - 2^{-2k +1}) B_{2k}(0) 
\end{split}
\end{align}

To present the optimal approximations to the Bernoulli periodic functions in the $L^1(\R/\Z)$-norm we recall the Euler polynomials $E_k(x)$ given by the generating function
\begin{equation}\label{EP}
\dfrac{e^{xt}}{e^t +1} = \dfrac{1}{2} \sum_{k=0}^{\infty} E_k(x)\dfrac{t^k}{k!}
\end{equation}
It will also be useful to have the Fourier expansions of the Euler periodic functions
\begin{eqnarray}
E_{2k}(x - [x]) &=& (-1)^k \,\,\, \dfrac{4(2k)!}{\pi^{2k+1}}\, \, \sum_{v=0}^{\infty} \dfrac{\sin(2v+1)\pi x}{(2v+1)^{2k+1}} \label{Eulereven} \\
E_{2k+1}(x - [x]) &=& (-1)^{k+1} \,\,\, \dfrac{4(2k+1)!}{\pi^{2k+2}} \,\, \sum_{v=0}^{\infty} \dfrac{\cos(2v+1)\pi x}{(2v+1)^{2k+2}} \label{Eulerodd} 
\end{eqnarray}
Define the sequence $(\theta_n)_{n \in \N_0}$ by
\begin{equation}\label{Intro4}
\theta_n = \left\{
\begin{array}{cc}
 0 & \textrm{if} \ \ n \ \ \textrm{is even}\\
1/2 & \textrm{if} \ \ n \ \ \textrm{is odd}
\end{array}
\right.
\end{equation}
Our second result is
\begin{theorem}\label{thm3.4}
Let $n$ and $N$ be nonnegative integers. There exists a trigonometric polynomial $x \mapsto R_{n+1}(x;N)$ of degree at most $N$ such that for any trigonometric polynomial $W(x)$ of degree at most $N$ the inequality
\begin{equation}\label{trig2}
\int_{\R/\Z} |W(x) - \mathcal{B}_{n+1}(x)|\, \dx \geq  \frac{|E_{n+1}(\theta_n)|}{(2N+2)^{n+1}}
\end{equation}
holds, with equality if and only if $W(x) = R_{n+1}(x;N)$. 
\end{theorem}
The extremal trigonometric polynomial $R_{n+1}(x;N)$ is explicitly described in Section 3 (equations (\ref{Sec3.1}) - (\ref{Sec3.1.2})).\\

Since the Bernoulli polynomial $B_n(x)$ is monic of degree $n$, by simple linear algebra, we can use the approximations obtained in Theorems \ref{thm2.3} and \ref{thm3.4} to majorize, minorize and approximate in $L^1(\R/\Z)$ any polynomial periodic function
\begin{equation*}
f(x) = a_nx^n + ... + a_1x + a_0 \ \ \ \textrm{for} \ \ \ x\in [0,1)
\end{equation*}
In general, these approximations will not be extremal. An interesting case, for example, arises from the Bernoulli inversion formula
\begin{equation}\label{xn}
 x^n = \dfrac{1}{(n+1)}\sum_{k=0}^{n}\binom{n+1}{k}B_k(x)
\end{equation}
Substituting $B_k(x)$ in expression (\ref{xn}) by $P_k(x;N,\beta_{k-1}), P_k(x;N, \alpha_{k-1}), R_k(x;N)$ according to Theorems \ref{thm2.3} and \ref{thm3.4} we obtain, respectively, trigonometric polynomials of degree $N$ that minorize, majorize and approximate in $L^1(\R/\Z)$ the periodic function $f(x) = x^n$, $x\in [0,1)$ (we adopt above $B_0 = P_0 = R_0 = 1$).\\

We proceed now to the proofs of Theorems \ref{thm2.3} and \ref{thm3.4}. In this paper we identify functions defined on $\R$ having period $1$ with functions defined on the compact quotient group $\R/\Z$. Integrals over $\R/\Z$ are with respect to the Haar measure normalized so that $\R/\Z$ has measure $1$. We write $e(x) = e^{2\pi i x}$. The signum symmetric function $\sgn(x)$ is given by $\sgn(x) = 1$, if $x>0$, $\sgn(x) = -1$, if $x<0$ and $\sgn(0) = 0$. We denote $\sgn_+(x)$ as the right-continuous signum function (i.e. $\sgn_+(0) = 1$).

%%%%%%%%%%%%%%%%%%%%%%%%%%%%%%%%%%%%%%%%%%%%%%%%%%%%%%%%%%%%%%%%%%%%%%%%%%%%%%%%%%%%%%%%%%%%%%%%%%%%%%%%%%%%%%%%%%%%%%%%%%%%%%

\section{Proof of Theorem \ref{thm2.3}}%%%%%%%%%%section2%%%%%%%%%%%%%%%%%%%%%%%%%%%%%%%%%%%%%%%%
We start this section recalling the notation and results from \cite{Lit} that will be used here. For $0\leq \alpha \leq 1$, $z \in \C$ and $n \in \N_0 = \N \cup \{0\}$, we define the following entire functions of exponential type $2\pi$:
\begin{eqnarray}\label{candidates}
H_n(z;\alpha) = \left(\tfrac{\sin \pi(z-\alpha)}{\pi}\right)^2 \left\{ z^n \sum_{k=-\infty}^{\infty} \tfrac{\sgn_+(k)}{(z-k-\alpha)^2} + 2\sum_{k=1}^{n} B_{k-1}(\alpha) z^{n-k} + \tfrac{2B_n(\alpha)}{z-\{\alpha\}}\right\} &&
\end{eqnarray}
where $\{ \alpha \}$ denotes the fractional part of $\alpha$ and $B_n(x)$ is the $n$-th Bernoulli polynomial defined in (\ref{Bernoulli polynomials}). Recall the sequences $\{\alpha_n\}_{n \in \N_0}$ and $\{\beta_n\}_{n \in \N_0}$ defined in (\ref{Intro3}).
\begin{lemma}[cf. Theorem 1 of \cite{Lit}]
Let $n \in \N_0$. The inequality
\begin{equation}\label{Sec2.2}
\delta^{-n}H_n(\delta x; \alpha_n) \leq \sgn (x) x^n \leq \delta^{-n}H_n(\delta x; \beta_n)
\end{equation}
holds for all $x \in \R$. These are the unique extremals of exponential type $2\pi\delta$ and they satisfy
\begin{eqnarray}
\int_\R(\delta^{-n}H_n(\delta x; \beta_n) - \sgn(x) x^n )\, \dx &=& -\frac{2 B_{n+1}(\beta_n)}{(n+1)\delta^{n+1}} \label{int1}\\
\int_\R( \sgn(x) x^n - \delta^{-n}H_n(\delta x; \alpha_n) )\, \dx &=& \frac{2B_{n+1}(\alpha_n)}{(n+1)\delta^{n+1}}.\label{int2}
\end{eqnarray}
\end{lemma}

\begin{lemma}[cf. Lemmas 1 and 2 of \cite{Lit}]\label{lem2.2}
Define the functions
\begin{equation}\label{Sec2.3}
d_n(x; \delta, \alpha) = \delta^{-n}H_n(\delta x; \alpha) - \sgn(x) x^n
\end{equation}
For any $\delta >0$ and $0 \leq \alpha \leq 1$ the function $x \mapsto d_n(x; \delta, \alpha)$ is $O(x^{-2})$ as $|x| \to \infty$ and therefore it is integrable. Its Fourier transform is the continuous function given by
\begin{eqnarray}
\widehat{d}_n(t; \delta, \alpha) &=& -2\delta^{-n-1}\sum_{k=0}^{\infty} \tfrac{B_{k+n+1}(\alpha)}{(k+1)!}\left( \tfrac{k+1}{k+n+1} - \tfrac{|t|}{\delta} \right) (-2\pi i (\tfrac{t}{\delta}))^k \nonumber\\
& &\ \  + \delta^{-n-1}\tfrac{B_n(\alpha)}{\pi i} \sgn(t) (e(-\{\alpha\}\tfrac{t}{\delta}) - 1) \ \ \textrm{if} \ \ |t|<\delta \label{Sec2.4}\\
\nonumber\\
\widehat{d}_n(t; \delta, \alpha) &=& -\frac{2.n!}{(2\pi i t)^{n+1}} \ \ \textrm{if} \ \ |t| \geq \delta. \label{Sec2.5}
\end{eqnarray}
 \end{lemma}
Let $N$ be a nonnegative integer. To describe the extremal trigonometric polynomials of degree at most $N$ that majorize and minorize the Bernoulli periodic function $\mathcal{B}_{n+1}(x)$ it will be convenient to use $\delta = N+1$. For $0\leq \alpha \leq 1$, we define the following family of trigonometric polynomials
\begin{equation}\label{Sec2.5.1}
P_{n+1}(x;N,\alpha) = \sum_{k = -N}^{N} \widehat{P}_{n+1}(k; N, \alpha)\, e(kx)
 \end{equation}
where the Fourier coefficients are given by
\begin{equation}\label{Sec2.6}
\widehat{P}_{n+1}(0; N,\alpha) = \frac{B_{n+1}(\alpha)}{(N+1)^{n+1}}
\end{equation}
and
\begin{equation}\label{Sec2.7}
\widehat{P}_{n+1}(k;N,\alpha) = -\left(\frac{n+1}{2}\right) \left( \widehat{d}_n(k; (N+1), \alpha) + \frac{2.n!}{(2\pi i k)^{n+1}}\right)
\end{equation}
for $k \neq 0$. We are now in position to prove Theorem \ref{thm2.3}.

\begin{proof}[Proof of Theorem \ref{thm2.3}]
The case $n=0$ was treated by Vaaler in \cite{V}, so we focus in the case $n\geq 1$. Recall that we are using here $\delta = N+1$. Observe initially that the Poisson summation formula
\begin{equation}\label{Sec2.9}
\sum_{l=-\infty}^{\infty} d_n(x+l; \delta, \beta_n) = \sum_{l=-\infty}^{\infty} \widehat{d}_n(l;\delta, \beta_n)e(lx)
\end{equation}
holds for all $x \in \R/\Z$. The reason for this is simple, from Lemma \ref{lem2.2} the function $x \mapsto d_n(x; \delta, \beta_n)$ is $O(x^{-2})$ as $|x| \to \infty$, therefore the left hand side of (\ref{Sec2.9}) is a continuous function. From (\ref{Sec2.5}) the Fourier series on the right hand side of (\ref{Sec2.9}) is absolutely convergent, and this suffices to establish the Poisson summation.

Using (\ref{Bernoulli periodic}), (\ref{int1}), (\ref{Sec2.5}), (\ref{Sec2.6}) and (\ref{Sec2.7}) we find that 
\begin{align}\label{Sec2.11}
\begin{split}
\tfrac{2}{(n+1)}\left(\mathcal{B}_{n+1}(x) - P_{n+1}(x;N,\beta_n) \right) & =  \sum_{l=-\infty}^{\infty} \widehat{d}_n(l;\delta, \beta_n)e(lx)\\
& = \sum_{l=-\infty}^{\infty} d_n(x+l; \delta, \beta_n) \geq 0
\end{split}
\end{align}
where the last inequality comes from (\ref{Sec2.2}) and (\ref{Sec2.3}). This proves that 
\begin{equation}\label{2.12}
 \mathcal{B}_{n+1}(x) \geq P_{n+1}(x;N,\beta_n) 
\end{equation}
for all $x \in \R/\Z$. To prove uniqueness recall from (\ref{candidates}) that $H(x,\beta_n)$ interpolates $\sgn(x)x^n$ at the points $\beta_n + m$, $m \in \Z$. From this we find that
\begin{equation}\label{Sec2.13}
d_n(x; \delta,\beta_n) = 0 \ \ \ \ \textrm{if}  \ \ \ \  x = \tfrac{\beta_n + m}{\delta},\ \ m \in \Z 
\end{equation}
From (\ref{Sec2.11}) and (\ref{Sec2.13}) we have the equalities
\begin{equation}\label{Sec2.14}
\mathcal{B}_{n+1}\left(\tfrac{\beta_n + m}{\delta}\right) = P_{n+1}\left(\left(\tfrac{\beta_n + m}{\delta}\right);N,\beta_n\right), \ \ m = 0,1,2,...,N.
\end{equation}
Suppose now that $Q(x)$ is a trigonometric polynomial of degree at most $N$ such that $\mathcal{B}_{n+1}(x) \geq Q(x)$ for all $x \in \R/\Z$. Using (\ref{Sec2.14}) we have
\begin{align}\label{Sec2.15}
 \begin{split}
\int_{\R/\Z} Q(x)\dx &= \tfrac{1}{\delta} \sum_{m=0}^{N} Q\left(\tfrac{\beta_n + m}{\delta}\right) \leq \tfrac{1}{\delta}                           \sum_{k=0}^{N} \mathcal{B}_{n+1}\left(\tfrac{\beta_n + m}{\delta}\right)  \\
                     & = \tfrac{1}{\delta} \sum_{k=0}^{N} P_{n+1}\left(\left(\tfrac{\beta_n + m}{\delta}\right);N,\beta_n\right)  =\int_{\R/\Z} P_{n+1}(x; N, \beta_n) \dx
\end{split}
\end{align}
which proves (\ref{Sec2.8.1}). If we have equality in (\ref{Sec2.15}) this means that for $m=0,1,2,...,N$ we have
\begin{equation}\label{Sec2.16}
Q\left(\tfrac{\beta_n + m}{\delta}\right) = \mathcal{B}_{n+1}\left(\tfrac{\beta_n + m}{\delta}\right) = P_{n+1}\left(\left(\tfrac{\beta_n + m}{\delta}\right);N,\beta_n\right)
\end{equation}
As $\mathcal{B}_{n+1}(x)$ is continuously differentiable at $\R/\Z - \{0\}$, equalities (\ref{Sec2.16}) imply that for at least $N$ values of $m = 0,1,2,...,N$ we have
\begin{equation}\label{Sec2.17}
Q'\left(\tfrac{\beta_n + m}{\delta}\right) = \mathcal{B}'_{n+1}\left(\tfrac{\beta_n + m}{\delta}\right) = P'_{n+1}\left(\left(\tfrac{\beta_n + m}{\delta}\right);N,\beta_n\right)
\end{equation}
The $2N+1$ conditions in (\ref{Sec2.16}) and (\ref{Sec2.17}) are sufficient to conclude that $Q(x) = P_{n+1}(x;N,\beta_n)$ (see \cite[Vol. II, page 23]{Z}). The proof for the majorizing case is very similar.
\end{proof}

%%%%%%%%%%%%%%%%%%%%%%%%%%%%%%%%%%%%%%%%%%%%%%- Section3%%%%%%%%%%%%%%%%%%%%%%%%%%%%%%%%%%%%%%%%%%%%%%%%%%%%%%%%%%%%%%%%%%%%%%
\section{Proof of Theorem \ref{thm3.4}}

Here we start by recalling the corresponding extremal problem in the real line, solved by F. Littmann in \cite{Lit2}. In this paper he finds the best $L^1(\R)$-approximation to the function $f(x) = x_+^n$ ($f(x) = x^n$ for $x \geq 0$ and $f(x) = 0 $ for $x<0$) by an entire function of exponential type $\pi\delta$. 

The following facts come from section 6 of \cite{Lit2}. Let $\phi = \Gamma'/\Gamma$, where $\Gamma$ is the Euler Gamma function, and $\alpha \in [0,1]$. Define the following functions of exponential type $\pi$:
\begin{equation*}
G_n(z;\alpha) = \left(\tfrac{\sin \pi(z - \alpha)}{\pi}\right) \, z^n \left( \phi\left(\tfrac{\alpha - z}{2}\right) - \phi(\alpha - z) + \log 2 - \dfrac{1}{2}\sum_{k=0}^{n} E_k(\alpha)z^{-k-1}\right)
\end{equation*}
where $E_k(x)$ are the Euler polynomials defined in (\ref{EP}). Also, recall the sequence $(\theta_n)_{n \in \N_0}$  defined in (\ref{Intro4}).

\begin{lemma}[cf. Theorem 6.2 of \cite{Lit2}]\label{lem3.1} Let $n \in \N_0$. For any entire function $A(z)$ of exponential type $\pi\delta$, the inequality
\begin{equation}\nonumber
 \int_{-\infty}^{\infty} |A(x) - x_+^n| \, \dx \geq \dfrac{|E_{n+1}(\theta_n)|}{(n+1)\delta^{n+1}}
\end{equation}
holds, with equality if and only if $A(z) = \delta^{-n}G_n(\delta z; \theta_n)$.
\end{lemma}

\begin{lemma}[cf. Lemma 5.1, Theorem 4.3 and proof of Theorem 6.2 of \cite{Lit2}] Define the functions 
\begin{equation}
\varphi_n(x; \delta) =  \delta^{-n}G_{n}(\delta x;\theta_n)- x_+^n
\end{equation}

\begin{enumerate}
\item[(i)] For any $\delta >0$ the function $\varphi_n(x; \delta)$ is $O(x^{-2})$ as $|x| \to \infty$ and its Fourier transform satisfies 
\begin{equation}\label{Sec3.0}
 \hat{\varphi}_n(t;\delta) = - \dfrac{n!}{(2 \pi i t)^{n+1}}  \ \ \ \textrm{if} \ \ \ |t| \geq \delta/2
\end{equation}
\item[(ii)] Regarding the sign of $\varphi_n(x; \delta)$ we have
\begin{eqnarray}
 (-1)^{k+1} \sgn (\sin \pi \delta x) &=& \sgn (\varphi_{2k}(x; \delta)) \label{sin}\\
(-1)^{k+1} \sgn (\cos \pi \delta x) &=& \sgn (\varphi_{2k+1}(x; \delta)) \label{cos}
\end{eqnarray}
\end{enumerate}
\end{lemma}

\begin{remark} It is possible to write an explicit formula for the Fourier transform $\hat{\varphi}_n(t;\delta)$ when $|t|<\delta/2$, as done in (\ref{Sec2.4}). For the sake of completeness we quote this result of F. Littmann (unpublished). Define the function
\begin{equation*}
 h_{\alpha}(z) = e^{\alpha z}(1 - e^z)^{-1} - \tfrac{1}{2}
\end{equation*}
then
\begin{equation}\label{FT3}
 \hat{\varphi}_n(t; 1) = - \dfrac{n!}{(2 \pi i t)^{n+1}} + h_{\theta_n}^{(n)}(-2 \pi i t) \ \ \ \textrm{for} \ \  |t|<1/2
\end{equation}
and in general 
\begin{equation}\label{FT3}
 \hat{\varphi}_n(t; \delta) = \delta^{-n-1} \hat{\varphi}_n(\tfrac{t}{\delta}; 1) \ \ \ \textrm{for} \ \  |t|<\delta/2
\end{equation}
\end{remark}

Let $N$ be a nonnegative integer. To describe the trigonometric polynomial of degree at most $N$ that best approximates the Bernoulli periodic function $\mathcal{B}_{n+1}(x)$ in the $L^1(\R/\Z)$-norm it will be convenient to use $\delta = 2N+2$. Define
\begin{equation}\label{Sec3.1}
R_{n+1}(x;N) = \sum_{k=-N}^{N} \widehat{R}_{n+1}(k;N)\,e(kx)
\end{equation}
where the Fourier coefficients are given by
\begin{equation}\label{Sec3.1.1}
\widehat{R}_{n+1}(0;N) = -(n+1)\,\widehat{\varphi}_{n}(0; 2N+2)
\end{equation}
and 
\begin{equation}\label{Sec3.1.2}
\widehat{R}_{n+1}(k;N) = -(n+1)\left(\widehat{\varphi}_{n}(k; 2N+2) + \frac{n!}{(2\pi i k)^{n+1}}\right)
\end{equation}

\begin{proof}[Proof of Theorem \ref{thm3.4}]
The case $n=0$ was done by Vaaler in \cite{V}, so we will work here with $n\geq 1$. Throughout this proof we use $\delta = 2N+2$. We can argue as in the beginning of the proof of Theorem \ref{thm2.3} to establish the Poisson summation formula at every point $x \in \R/\Z$
\begin{equation}\label{Sec3.2}
\sum_{l=-\infty}^{\infty} \varphi_n(x+l; \delta) = \sum_{l=-\infty}^{\infty} \widehat{\varphi}_n(l;\delta)e(lx)
\end{equation}
From (\ref{Bernoulli periodic}), (\ref{Sec3.0}) and (\ref{Sec3.1}) we find that 
\begin{equation}\label{Sec3.3}
R_{n+1}(x;N) - \mathcal{B}_{n+1}(x) = -(n+1)\sum_{l=-\infty}^{\infty} \widehat{\varphi}_n(l;\delta)e(lx)
\end{equation}
Therefore, from (\ref{Sec3.2}) and (\ref{Sec3.3}) we have
\begin{align}\label{Sec3.4}
 \begin{split}
  \int_{\R/\Z} |R_{n+1}(x;N) - \mathcal{B}_{n+1}(x)|\, \dx & = (n+1) \int_{\R/\Z} \left|\sum_{l=-\infty}^{\infty} \varphi_n(x+l; \delta)\right|\dx
 \end{split}
\end{align}
Using (\ref{sin}), (\ref{cos}) and Lemma \ref{lem3.1} we find that expression (\ref{Sec3.4}) is equal to
\begin{align}
 \begin{split}
  = (n+1) \int_{\R/\Z} \sum_{l=-\infty}^{\infty} \left| \varphi_n(x+l; \delta)\right|\dx  & = (n+1)\int_{\R}|\varphi_n(x; \delta)|\,\dx \\
& = \frac{|E_{n+1}(\theta_n)|}{\delta^{n+1}}
 \end{split}
\end{align}
and this proves that equality happens in (\ref{trig2}) when $W(x) = R_{n+1}(x;N)$. 

To prove uniqueness we divide the argument in two cases. Suppose first that $n$ is an even integer. As $\sgn(\sin \pi x)$ is a normalized function of bounded variation on $[0,2]$ its Fourier expansion
\begin{equation}\label{FS}
\sgn(\sin \pi x) = \dfrac{2}{\pi i} \sum_{k=-\infty}^{\infty} \dfrac{1}{(2k+1)}\, e( (k+\tfrac{1}{2})\,x)
\end{equation}
converges at every point $x$ and the partial sums are uniformly bounded. For a general trigonometric polynomial $W(x)$ of degree at most $N$ we have, by (\ref{FS}), (\ref{Bernoulli periodic}) and (\ref{Eulerodd})
\begin{align}
\begin{split}
\int_{\R/\Z} |W(x) &- \mathcal{B}_{n+1}(x)|\, \dx \geq \left |\int_{\R/\Z} (W(x) - \mathcal{B}_{n+1}(x)) \sgn \{\sin \pi \delta x\}\, \dx \right|\label{inequality} \\
& =  \left|\int_{\R/\Z} \mathcal{B}_{n+1}(x) \sgn \{\sin \pi \delta x\}\, \dx \right| \\
& =  \left| \dfrac{2}{\pi i} \sum_{k = -\infty}^{\infty} (2k +1)^{-1} \int_{\R/\Z} \mathcal{B}_{n+1}(x) e((k+\tfrac{1}{2})\delta x)\, \dx \right|  \\
& =  \frac{2 (n+1)!}{\pi^{n+2}\delta^{n+1}} \sum_{k=-\infty}^{\infty} \frac{1}{(2k+1)^{n+2}} = \frac{|E_{n+1}(\theta_n)|}{\delta^{n+1}} 
\end{split}
\end{align}
which proves (\ref{trig2}). If equality happens in (\ref{inequality}) we must have
\begin{equation}
W\left(\tfrac{k}{2N+2}\right) = \mathcal{B}_{n+1}\left(\tfrac{k}{2N+2}\right) \ \ \ \textrm{for} \ \ \ k = 1,2,...,2N+1.
\end{equation}
Since the degree of $W(x)$ is at most $N$, such polynomial exists and is unique \cite[Vol. II, page 1]{Z}. It is not hard to see that $R_{n+1}(x;N)$ satisfies the same property (equations (\ref{sin}), (\ref{Sec3.2}) and (\ref{Sec3.3})), so we must have $ R_{n+1}(x;N)= W(x)$. The proof for $n$ odd integer follows the same ideas using (\ref{Eulereven}), (\ref{cos}) and changing $x$ by $x + 1/2$ in (\ref{FS}). 
\end{proof}

%%%%%%%%%%%%%%%%%%%%%%%%%%%%%%%%%%%%%%%%%%%%%%%%%%%%%%%%%%%Section 4%%%%%%%%%%%%%%%%%%%%%%%%%%%%%%%%%%%%%%%%%%%%%%%%%%%%%%%%%
\section{Erd\"{o}s-Tur\'{a}n inequalities}
Let $x_1, x_2, \dots , x_M$ be a finite set of points in $\R/\Z$.  A basic problem in the theory
of equidistribution is to estimate the discrepancy of the points $x_1, x_2, \dots , x_M$ by an expression
that depends on the Weyl sums
\begin{equation}\label{et0}
\sum_{m=1}^M e(kx_m),\quad\text{where}\quad k = 1, 2, \dots , N.
\end{equation} 
This is most easily accomplished by using the sawtooth function $\Psi:\R/\Z\rightarrow \R$, defined in the Introduction by
\begin{equation*}\label{et1}
\Psi(x) = \begin{cases} x-[x]-\tfrac12 & \textrm{if} \ \ x\not\in \Z\\
            0 & \text{if} \ \ x \in \Z\end{cases}
\end{equation*}
where $[x]$ is the integer part of $x$.  A simple definition for the discrepancy of the finite set is 
\begin{equation*}\label{et2}
\Delta_M(\bx) = \sup_{y\in\R/\Z} \Bigl|\sum_{m=1}^M \Psi\bigl(x_m - y\bigr)\Bigr|.
\end{equation*}
In this setting the Erd\"{o}s-Tur\'{a}n inequality is an upper bound for $\Delta_M$ of the form
\begin{equation}\label{et3}
\Delta_M(\bx) \le c_1 MN^{-1} + c_2 \sum_{k=1}^N k^{-1} \Bigl|\sum_{m=1}^M e(kx_m)\Bigr|,
\end{equation}
where $c_1$ and $c_2$ are positive constants.  In applications to specific sets the parameter $N$ can be 
selected so as to minimize the right hand side of (\ref{et3}).  Bounds of this kind follow easily from knowledge 
of the extremal trigonometric polynomials that majorize and minorize the function $\Psi(x)$.  This is discussed in
\cite{CV}, \cite{ET}, \cite{M2}, \cite{V}, and \cite{V2}.  An extension to the spherical cap discrepancy is derived 
in \cite{LV}, and a related inequality in several variables is obtained in \cite{BMV}.\\
\\
As already noted in the Introduction of this paper, the sawtooth function $\Psi(x)$ coincides with the first Bernoulli periodic function $\mathcal{B}_1(x)$. One is naturally led to generalize the concept of discrepancy using the other Bernoulli functions. For $n\geq 0$ define
\begin{equation}\label{Sec4.1}
\Delta_M^{n+1}(\bx) = \sup_{y\in\R/\Z} \Bigl|\sum_{m=1}^M \mathcal{B}_{n+1}\bigl(x_m - y\bigr)\Bigr|
\end{equation}
We recall the extremal trigonometric polynomials of degree at most $N$ given by Theorem \ref{thm2.3}
\begin{equation}\label{Sec4.2}
P_{n+1}(x;N,\beta_n) \leq \mathcal{B}_{n+1}(x) \leq P_{n+1}(x;N,\alpha_n) \vspace{0.3cm}
\end{equation}
The following bound for the generalized discrepancy $\Delta_M^{n+1}(\bx)$ will follow from (\ref{Sec4.2}) and algebraic manipulations.
\begin{proposition}
Let $\bx = (x_1,x_2,...,x_M)$ be a sequence of numbers in $\R/\Z$. Then
\begin{align}\label{Sec4.5}
 \begin{split}
\Delta_{M}^{n+1}(\bx) \leq \max & \left\{ - \tfrac{M B_{n+1}(\beta_n)}{(N+1)^{n+1}} + \sum_{0 <|k| \leq N} \left|\widehat{P}_{n+1}(k;N,\beta_n)\right| \left|\sum_{m=1}^M e(x_m k)\right| \right.,\\
& \left. \tfrac{M B_{n+1}(\alpha_n)}{(N+1)^{n+1}} + \sum_{0 <|k| \leq N} \left|\widehat{P}_{n+1}(k;N,\alpha_n)\right| \left|\sum_{m=1}^M e( x_m k)\right| \right\} 
 \end{split}
\end{align}
\end{proposition}

\begin{proof}
Let $y \in \R/\Z$. From (\ref{Sec4.2}) we have

\begin{align*}
\begin{split}
 \sum_{m=1}^M P_{n+1}(x_m-y;N,\beta_n) & \leq \sum_{m=1}^M \mathcal{B}_{n+1}\bigl(x_m - y\bigr) \leq \sum_{m=1}^M P_{n+1}(x_m-y;N,\alpha_n)
\end{split}
\end{align*}
which implies that 
\begin{align}\label{Sec4.6}
\begin{split}
 \sum_{m=1}^M \sum_{|k| \leq N} \widehat{P}_{n+1}(k;N,\beta_n)\, & e(k(x_m - y))  \leq \sum_{m=1}^M \mathcal{B}_{n+1}\bigl(x_m - y\bigr)  \\
&\leq \sum_{m=1}^M \sum_{|k| \leq N} \widehat{P}_{n+1}(k;N,\alpha_n)\, e(k(x_m - y)) .
\end{split}
\end{align} 
Interchanging the sums in (\ref{Sec4.6}) we get
\begin{align}\label{Sec4.7}
\begin{split}
& M\widehat{P}_{n+1}(0;N,\beta_n) + \sum_{0< |k| \leq N} \widehat{P}_{n+1}(k;N,\beta_n)\,  \sum_{m=1}^M e(k(x_m - y))  \\
\leq & \, \sum_{m=1}^M \mathcal{B}_{n+1}\bigl(x_m - y\bigr)  \\
\leq & \, M\widehat{P}_{n+1}(0;N,\alpha_n) + \sum_{0 <|k| \leq N} \widehat{P}_{n+1}(k;N,\alpha_n)\,  \sum_{m=1}^M e(k(x_m - y))
\end{split}
\end{align} 
and from (\ref{Sec4.7}) we conclude that 
\begin{align}\label{Sec4.8}
\begin{split}
& - \left|M\widehat{P}_{n+1}(0;N,\beta_n)\right| - \sum_{0< |k| \leq N} \left|\widehat{P}_{n+1}(k;N,\beta_n)\right|\,  \left|\sum_{m=1}^M e(k(x_m - y))\right|  \\
\leq & \, \sum_{m=1}^M \mathcal{B}_{n+1}\bigl(x_m - y\bigr)  \\
\leq & \, \left|M\widehat{P}_{n+1}(0;N,\alpha_n)\right| + \sum_{0 <|k| \leq N} \left|\widehat{P}_{n+1}(k;N,\alpha_n)\right|\,  \left|\sum_{m=1}^M e(k(x_m - y))\right|.
\end{split}
\end{align} 
Expression (\ref{Sec2.6}) gives us
\begin{equation*}
\left|\widehat{P}_{n+1}(0;N,\beta_n)\right| = - \tfrac{ B_{n+1}(\beta_n)}{(N+1)^{n+1}} \ \ \ \textrm{and} \ \ \ \left|\widehat{P}_{n+1}(0;N,\alpha_n)\right| =  \tfrac{ B_{n+1}(\alpha_n)}{(N+1)^{n+1}}.
\end{equation*}
This fact allied to the equality
\begin{equation*}
\left|\sum_{m=1}^M e(k(x_m - y))\right| = \left|\sum_{m=1}^M e(kx_m)\right|
\end{equation*}
show that (\ref{Sec4.8}) implies the desired bound (\ref{Sec4.5}).

\end{proof}

For applications, it would be desirable to obtain simple bounds for the Fourier coefficients $\widehat{P}_{n+1}(k;N,\alpha)$. Another question that arises here is: are there any interesting inequalities relating the discrepancies $\Delta_{M}^{n}$ and $\Delta_{M}^{n+1}$?

\section{Bounds for Hermitian forms}%%%%%%%%%%%%%%%%%%%%%%%%%Section6%%%%%%%%%%%%%%%%%%%%%%%%%%%%%%%%%%%%%%%%%%%%%%%%%%%%

A classical application of the theory of extremal functions of exponential type provides sharp bounds for some Hilbert-type inequalities. In \cite{Lit}, F. Littmann obtained the following result (recall the Bernoulli polynomials $B_n(x)$ and the sequences  $\{\alpha_n\}_{n \in \N_0}$ and $\{\beta_n\}_{n \in \N_0}$ defined in the Introduction).

\begin{proposition}[cf. Corollary 2 of \cite{Lit}]\label{prop11}
Let $\{\lambda_r\}_{r=1}^{N}$ be a sequence of well-spaced real numbers, i.e. $|\lambda_r - \lambda_s| \geq \delta$ for all $r \neq s$. Let $\{a_r\}_{r=1}^N$ be a sequence of complex numbers and $m \in \N$. We have
\begin{equation}\label{Sec6.1}
  -L_m(\delta) \sum_{r=1}^{N} |a_r|^2 \leq \sum_{\stackrel{r,s=1}{r\neq s}}^{N} \frac{a_r \overline{a_s}}{(i(\lambda_r- \lambda_s))^m} \leq   U_m(\delta) \sum_{r=1}^{N} |a_r|^2
\end{equation}
with the optimal constants
\begin{equation}
L_m(\delta) = (2\pi)^m \frac{B_m(\alpha_{m-1})}{m! \delta^m} \ \ \ \textrm{and} \ \ \ U_m(\delta) =   -(2\pi)^m \frac{B_m(\beta_{m-1})}{m! \delta^m} \nonumber
\end{equation}
\end{proposition}
There is a simple argument, due to H.L. Montgomery (see Corollary 1 of \cite{M}), that allows us to pass inequalities (\ref{Sec6.1}) to periodic versions. For this we define the periodic functions $p_m: \R/\Z-\{0\} \to \R$ and  $q_m: \R/\Z-\{0\} \to \R$ by
\begin{eqnarray}
p_m(x) &=& \sum_{k \in \Z} \frac{1}{(x+k)^m}  \\%= \frac{(\pi \cot \pi x)^{(m-1)}(x)}{(-1)^{m-1} (m-1)!}
q_m(x) &=& \sum_{k \in \Z} \frac{(-1)^k}{(x+k)^m}
\end{eqnarray}
For real numbers $x$ we write 
\begin{equation}
\|x\| = \textrm{min}\{ |x-m|: m \in \Z\} \nonumber
\end{equation}
for the distance from $x$ to the nearest integer. We have the following 
\begin{proposition}\label{prop12}
Let $\{\lambda_r\}_{r=1}^{N}$ be a sequence of well-spaced real numbers in $\R/\Z$, i.e. $\|\lambda_r - \lambda_s\| \geq \delta$ for all $r \neq s$. Let $\{a_r\}_{r=1}^N$ be a sequence of complex numbers and $m \in \N$. If $m$ is odd we have
\begin{equation}\label{Sec6.2}
-L_m(\delta) \sum_{r=1}^{N} |a_r|^2  \leq  \displaystyle\sum_{\stackrel{r,s=1}{r\neq s}}^{N} i^{-m} \,a_r \overline{a_s} \, p_m(\lambda_r - \lambda_s)  \leq   U_m(\delta) \sum_{r=1}^{N} |a_r|^2 
\end{equation}
and 
\begin{equation}\label{Sec6.2.1}
-L_m(\delta) \sum_{r=1}^{N} |a_r|^2  \leq  \displaystyle\sum_{\stackrel{r,s=1}{r\neq s}}^{N} i^{-m} \,a_r \overline{a_s} \, q_m(\lambda_r - \lambda_s)  \leq   U_m(\delta) \sum_{r=1}^{N} |a_r|^2
\end{equation}
If $m$ is even we have
\begin{align}\label{Sec6.3}
\begin{split}
-\left(2i^{-m}\zeta(m)+ L_m(\delta)\right) \sum_{r=1}^{N} |a_r|^2  &\leq \displaystyle\sum_{\stackrel{r,s=1}{r\neq s}}^{N} i^{-m} \,a_r \overline{a_s} \, p_m(\lambda_r - \lambda_s) \\
& \leq  \left(-2i^{-m}\zeta(m) + U_m(\delta)\right) \sum_{r=1}^{N} |a_r|^2 
\end{split}
\end{align}
and
\begin{align}\label{Sec6.3.1}
\begin{split}
\left((2 -2^{2-m})i^{-m}\zeta(m) -L_m(\delta)\right) & \sum_{r=1}^{N} |a_r|^2   \leq  \displaystyle\sum_{\stackrel{r,s=1}{r\neq s}}^{N} i^{-m} \,a_r \overline{a_s} \, q_m(\lambda_r - \lambda_s)  \\
& \leq   \left((2 -2^{2-m})i^{-m}\zeta(m) + U_m \right)(\delta) \sum_{r=1}^{N} |a_r|^2
\end{split}
\end{align}
where $\zeta$ is the Riemann zeta function.
\end{proposition}
\begin{proof}
We prove here inequality (\ref{Sec6.3}). Apply Proposition \ref{prop11} with a doubly-indexed set of $NK$ variables $a_{rj}$, $1\leq r\leq N$, $1\leq j \leq K$ and well-spaced constants $\lambda_{rj}$. Then

\begin{equation}
  -L_m(\delta) \sum_{r,j} |a_{rj}|^2 \leq \sum_{\stackrel{r,s,j,\,l}{(r,j)\neq (s,l)}} \frac{a_{rj} \overline{a_{sl}}}{(i(\lambda_{rj}- \lambda_{sl}))^m} \leq   U_m(\delta) \sum_{r,j} |a_{rj}|^2  \nonumber
\end{equation}
Now put $a_{rj} = a_r$ and $\lambda_{rj} = \lambda_r + j$. Then
\begin{align}\label{Sec6.4}
\begin{split}
 -KL_m(\delta) \sum_{r} |a_{r}|^2 & \leq \
\sum_{r\neq s} i^{-m} a_{r} \overline{a_{s}} \sum_{j,\,l} (\lambda_r - \lambda_s + j - l)^{-m} \\
&+ i^{-m}\sum_{r}|a_r|^2 \sum_{j\neq l} (j-l)^{-m} 
 \leq KU_m(\delta) \sum_{r} |a_{r}|^2
\end{split}
\end{align}
Calling $j-l = k$ and dividing (\ref{Sec6.4}) by $K$ we obtain
\begin{align} \nonumber
\begin{split}
-L_m(\delta) \sum_{r} |a_{r}|^2 & \leq \sum_{r\neq s} i^{-m} a_{r} \overline{a_{s}} \sum_{k=-K}^K (1-|k|/K) (\lambda_r - \lambda_s + k)^{-m} \\
&+ i^{-m}\sum_{r}|a_r|^2 \sum_{\stackrel{k=-K}{k\neq 0}}^K (1-|k|/K)k^{-m} 
\leq U_m(\delta) \sum_{r} |a_{r}|^2
\end{split}
\end{align}
Now it is just a matter of sending $K \to \infty$ to obtain
\begin{align} \nonumber
\begin{split}
-L_m(\delta) \sum_{r} |a_{r}|^2 & \leq \sum_{r\neq s} i^{-m} a_{r} \overline{a_{s}} p_m(\lambda_r - \lambda_s) + 2 i^{-m}\zeta(m) \sum_{r}|a_r|^2  \\
& \leq U_m(\delta) \sum_{r} |a_{r}|^2
\end{split}
\end{align}
and this proves (\ref{Sec6.3}). To prove (\ref{Sec6.3.1}) we put $a_{rj} = (-1)^ja_r$ and repeat the argument. When $m$ is odd, the proofs of (\ref{Sec6.2}) and (\ref{Sec6.2.1}) are even simpler when we notice that
\begin{equation} \nonumber
\sum_{\stackrel{k=-K}{k\neq 0}}^K (1-|k|/K)k^{-m} = \sum_{\stackrel{k=-K}{k\neq 0}}^K (-1)^k(1-|k|/K)k^{-m} = 0
\end{equation}
\end{proof}
Using the identities 
\begin{equation}
q_1(x) = \frac{\pi}{\sin{\pi x}} \ \ , \ \ p_1(x) = \frac{\pi}{\tan \pi x} \ \ , \ \  p_2(x) = \frac{\pi^2}{\sin^2 \pi x}
\end{equation}
and relations (\ref{Riemann}) we obtain the following interesting special cases
\begin{corollary}
In the hypotheses of Proposition \ref{prop12} we have
\begin{eqnarray}
\left| \displaystyle\sum_{\stackrel{r,s=1}{r\neq s}}^N \frac{a_r \overline{a_s}}{\sin \pi (\lambda_r - \lambda_s)}\right| & \leq & \frac{1}{\delta} \sum_{r=1}^{N} |a_r|^2\ \ , \label{Sec6.5}\\
\left| \displaystyle\sum_{\stackrel{r,s=1}{r\neq s}}^N \frac{a_r \overline{a_s}}{\tan \pi (\lambda_r - \lambda_s)}\right| & \leq & \frac{1}{\delta} \sum_{r=1}^{N} |a_r|^2
\end{eqnarray}
and
\begin{equation}
- \tfrac{1}{6}\left(\tfrac{1}{\delta^2} + 2 \right) \sum_{r=1}^{N} |a_r|^2 \leq \displaystyle\sum_{\stackrel{r,s=1}{r\neq s}}^N \frac{a_r \overline{a_s}}{\sin^2 \pi (\lambda_r - \lambda_s)} \leq \tfrac{1}{3}\left(\tfrac{1}{\delta^2} - 1 \right)\sum_{r=1}^{N} |a_r|^2
\end{equation}
\end{corollary}
Expression (\ref{Sec6.5}) was used by H.L. Montgomery in \cite{M} as a step in the proof of the large sieve inequality.

\section*{Acknowledgments}%%%%%%%%%%%%%%%%%%%%%%%%%%%%%%%%%%%%%%%%%%%%%%%%%%%%%%%%%%%%%%%%%%%%%%%%%%%%%%%%%%%%%%%%%%%%%%
I would like to thank Jeffrey Vaaler for all the motivation and insightful conversations on this subject. I would also like to thank the referee of this paper for the valuable suggestions. It has come to my attention, thanks to Michael Ganzburg, that results similar to Theorems 1 and 2 already appear in the literature with different proofs (see \cite{AK}, \cite{F}
and \cite{Ga}).

\end{document}